\documentclass{amsart}
\usepackage{tikz}
\usepackage{color,soul,amssymb}

\title[Symmetry groups of Boolean functions: simple groups]{Symmetry groups of Boolean functions: \\ simple groups
}

\author{Mariusz Grech, Andrzej Kisielewicz}
\address{Institute of Mathematics, University of Wroclaw \\
pl.Grunwaldzki 2, 50-384 Wroclaw, Poland}
\email{[Mariusz.Grech,Andrzej.Kisielewicz]@math.uni.wroc.pl}
\thanks{Supported in part by Polish NCN grant 2016/21/B/ST1/03079}

\keywords{Symmetry groups of Boolean functions, automorphism group of hypergraphs, relation groups, permutation groups, simple groups, regular set, distinguishing number.}

\begin{document}

\newtheorem{Theorem}{Theorem}[section]
\newtheorem{Lemma}[Theorem]{Lemma}
\newtheorem{Fact}[Theorem]{Fact}
\newtheorem{Proposition}[Theorem]{Proposition}
\newtheorem{Corollary}[Theorem]{Corollary}

\newcommand{\G}{\mathcal G}
\newcommand{\f}{F}

\renewcommand{\to}{\rightarrow}
\renewcommand{\gg}{g}
\newcommand{\hh}{h}
\newcommand{\uu}{u}
\newcommand{\ww}{w}
\newcommand{\kk}{k}
\newcommand{\vv}{v}
\newcommand{\bs}{\backslash}

\renewcommand{\iff}{\longleftrightarrow}


\begin{abstract}
We consider the problem of characterizing the class of those permutation groups that are the symmetry groups of Boolean functions. These are exactly the \emph{automorphism groups of hypergraphs}. They are also called the \emph{relation groups}. In this paper we describe those of them that are simple as abstract groups. This is done by combining results based on the classification of finite simple groups with the description of intransitive actions of simple groups.
We also obtain a complete characterization of those simple permutation groups that have regular sets, and prove that (with one exception) if a simple permutation group $G$ is a relation group, then every subgroup of $G$ is a relation group.
\end{abstract}

\maketitle

\section{Introduction}

It is a well known result by Frucht that every group is isomorphic with the automorphism group of some graph, which implies easily that every group is isomorphic with the symmetry group of some Boolean function. However, this fact is far away from knowledge that is needed in many applications of the automorphism groups of various combinatorial structures. Often knowledge on orbits of these groups, dependencies between actions on different orbits, and generally, the combinatorial structure of the corresponding permutation groups is needed (see, for example, problems in computer science and universal algebra mentioned in \cite{CK,kis1}). Also, if one wishes to classify combinatorial structures \emph{via} their automorphism groups, then using abstract groups is a very rough classification that ignores the fact that only highly symmetric combinatorial structures enjoy the property of transitivity of elements. For example, most graphs are not vertex transitive and the action of the automorphism group on different orbits may be very different and essential with regard to various graph properties.

Meanwhile, most research in permutation group theory is focused on transitive groups, most notably in connection with the task of classifying the finite simple groups. Knowledge on intransitive actions is incidental and scattered. This is our general motivation to study permutation groups on their own, and in particular, to study the actions of imprimitive and intransitive permutation groups.

One of the natural questions in this direction is which permutation groups can be represented as the symmetry groups of Boolean functions. This question is strictly connected with applications in computer science concerning parallel complexity of languages (see \cite{CK}). Yet, describing the symmetry groups of Boolean function may serve also as a general tool for classifying permutation groups with regard to their combinatorial complexity.

The \emph{symmetry group} of a Boolean function $\f:\{0,1\}^n \to \{0,1\}$ is the set of all permutations $\gg$ such that
$$
\f(x_{1\gg}, x_{2\gg},\ldots,x_{x\gg}) = \f(x_1,x_2,\ldots,x_n)
$$
for all $x_1,x_2,\ldots,x_n \in \{0,1\}$.
The class of all permutation groups that are the symmetry groups of Boolean function is denoted $BGR(2)$. The analogous notation $BGR(k)$ stands for the class of all permutation groups that are the symmetry groups of $k$-valued Boolean function, i.e., those that have the arguments in the set $\{0,1\}$, but take values in the set $\{0,1,\ldots, k-1\}$. Finally, $BGR$ is the union of all $BGR(k)$ for any $k\geq 2$.

Considering $n$-tuples $(x_1,x_2,\ldots,x_n)$ as the indicator functions of subsets of the set of indices, each Boolean function $\f$ on $n$ variables may be considered as a function $\f(x)$ from the family of the subsets $x\subseteq \{1,2,\ldots,n\}$ to $\{0,1\}$, and this is the way we use in this paper. This yields the correspondence of the Boolean functions with hypergraphs (defined as families of subsets), and shows that the class $BGR(2)$ may be viewed as the class of the \emph{automorphism groups of
hypergraphs}. A family of subsets may be also viewed as an unordered relation, and the symmetry group of the corresponding Boolean function as the invariance group of such a relation (see \cite{SV}). The classes $BGR(k)$ correspond to the automorphism groups of colored hypergraphs or relational structures. 
In particular, $BGR$ is exactly the class of permutation groups that in terms of \cite{SV} are \emph{orbit closed} (see also \cite{ing,SY,SW} for other related work on this kind of orbit equivalence).

Although every abstract group is isomorphic with the symmetry group of a Boolean function, not every permutation group occurs as such. For example the alternating group $A_n$, in its natural action on the $n$-element set, is not the symmetry group of any Boolean function. Yet, there are Boolean function whose symmetry groups are abstractly isomorphic with $A_n$, but not permutation isomorphic. In this paper, in particular, all such permutation groups are described. 

In \cite{CK}, as a key result, a theorem was stated that actually $BGR=BGR(2)$, i.e. that the symmetry groups of Boolean functions are exactly the orbit closed permutation groups.
In \cite{kis1}, this result has been shown to be false. It was observed that the Klein four group $K_4$ in its natural action belongs to $BGR(3)$, but not to $BGR(2)$. So far, no other counterexample like this has been found, and no result showing that this is an exceptional example has been proved. We observe, that the claim in \cite[Corollary~5.3]{SV} that there are infinitely many wreath products providing such counterexamples also turned out to be false (the proof provides, in fact, no single example like this; more details on that are given in \cite{GKwr}). So, the intriguing problem whether there are only a few exceptions to the claim that $BGR=BGR(2)$ or whether the classes $BGR(k)$ contain many unknown permutation groups, remains open.

In \cite{GK1}, we have proved that every abelian permutation group, but known exceptions, is the symmetry group of a Boolean function. In fact the result included a larger class of permutation groups, namely, all subgroups of direct sums of regular permutation groups.
In this paper we consider this problem for the other extreme: the class of those permutation groups that are simple as abstract groups. The fact that these groups have no nontrivial normal subgroups, on the one hand, and the classification of finite simple groups, on the other hand, make it feasible to obtain a complete description for this class.

It follows from the results of \cite{SV} that the only primitive
simple groups, other than $A_n$, that are not the symmetry groups of Boolean functions are $PSL(2,8)$ and $C_5$ in their natural actions.
Yet, simple groups have also many imprimitive and intransitive actions, and we wish to establish which of the corresponding imprimitive and intransitive permutation groups are the symmetry groups of Boolean functions, and which are not. In particular, our first motivation was to find other, more complex counterexamples to the claim that $BGR=BGR(2)$. We prove, however, that in the class of simple groups the claim turns out to be true.

The paper is organized as follows. We start with recalling the basic notions in Section~\ref{s:pre}. In the next section we deal with transitive imprimitive simple groups. We show that all such groups occur as the symmetry groups of Boolean functions.
Intransitive actions of simple groups are described in Proposition~\ref{p:simple}. This is relatively easy to do using the description of the general construction of intransitive permutation groups contained in \cite{kis1}. This construction and some basic results are recalled for the reader convenience in Section~\ref{s:ss}. Next, using the classification of the simple groups and results from \cite{BP,kis1,ser,SV} we prove a number of lemmas on intransitive simple groups,
leading to our main result, which is given in Theorem~\ref{th:main}. It provides a complete description of all simple permutations groups that are the symmetry groups of Boolean functions. All the remaining simple permutation groups, which includes many intransitive groups abstractly isomorphic to alternating groups, are proved not to belong to $BGR$, at all. Our research yields also additional results characterizing simple permutation groups that have a regular set (Theorem~\ref{th:reg}), and showing that (with one exception) the property that a simple permutation groups is the symmetry group of a Boolean function is inherited by subgroups (Theorem~\ref{th:sub}).

Throughout, many particular cases arise that requires some computer calculations. We have used system GAP 4.10.1. In cases of applying GAP, we try to give only those computational details, that seem sufficient for the reader to verify our results, using any computational system for permutation groups.

\section{Prerequisites}\label{s:pre}

A \emph{permutation group} $G$ is a subgroup of the symmetric group $Sym(\Omega)$ on a set $\Omega$.
In this paper, all permutation groups are finite and considered up to \emph{permutation isomorphism} \cite[p.~17]{DM} (i.e., two groups that differ only in labelling of points are treated as the same).
A permutation group $G$ is called \emph{trivial}, if the only permutation in $G$ is the identity. 
The symmetric group on $n$ elements is denoted $S_n = Sym(\{1,2,\ldots,\})$, and the alternating group on $n$ elements is denoted $A_n$. The cyclic group generated by the permutation $(1,2,\ldots,n)$ is denoted $C_n$. We use standard notation for abstract groups (used in \cite{DM}) as the notation for permutation groups in their natural actions. The projective special linear groups are denoted $PSL(d,q)$ or $L_d(q)$ (depending on the source we refer to). We recall that in their natural actions they are permutation groups on $n= (q^d-1)/(q-1)$ elements.

Following \cite{cam}, we write permutations on the right, and compose from left to right (as in the notation $\alpha(\gg\hh) = (\alpha\gg)\hh $). The same notation is applied to the action of permutations on subsets or pairs of elements of the underlying set $\Omega$.
Yet, the orbit of a set $x \subseteq \Omega$ in the action of $G \leq Sym(\Omega)$ on the subsets of $\Omega$ is denoted $x^G$.

We view Boolean functions $\f(x)$ as the functions from the set $P(\Omega)$ of subsets of a set $\Omega$ to the set $\{0,1\}$. 
The family of sets $x$ with $\f(x)=1$ may be considered as an \emph{unordered relation} $R$ on $\Omega$, and the \emph{symmetry group} of $\f$ is the set of permutations of $\Omega$ preserving $R$ (cf. \cite{SV}). We will use the notation $G=\G(R)$ for this group, and refer to $R$ as a \emph{defining relation} for $G$.

More generally, for a $k$-valued Boolean function $\f$, the family ${\mathcal R} = (R_0,\ldots,R_{k-1})$ of relations $R_i$ defined by $\f(x)=i$ may be considered as a relational structure $P(\Omega, \mathcal R)$ (of unordered relations), and the symmetry group of $\f$
is the set of permutations of $\Omega$ preserving all the relations in ${\mathcal R}$.
Yet, in this case it is more convenient to speak of the symmetry group $\G(\f)$ of the function $\f: P(\Omega) \to \{0,\dots,k-1\}$. 

We should note that the groups defined by (ordered) relations have been introduced by Wielandt in \cite{wie}, and it must be emphasized that we are concerned merely with unordered relations and actions of permutation groups on (unordered) sets (as in \cite{kan,LW}).

The action of a group $G \leq Sym(\Omega)$
on the power set $P(\Omega)$ preserves the partition into subsets of different cardinalities, so it has at least $|\Omega|+1$ orbits. If $G=\G(R)$, then obviously $R$ must be a union of orbits of $G$ in its action on $P(\Omega)$, and this observation is used in looking for defining relations for a permutation group. Note that whether $\Omega \in R$ or not does not matter for the definition, so in a case of more involved constructions we may make an arbitrary assumption here. Similarly, if $G$ is transitive, then either all singletons belong to $R$ or no singleton belongs to $R$, and again we may assume the former case or the latter, depending on the need. The same concerns the complements of singletons in $\Omega$, because the action of $G$ on a set $x$ and $\Omega\setminus x$ is essentially the same. Note also that a permutation $u$ on $\Omega$ preserves $R$ if and only if it preserves the sets of cardinality $k$ in $R$, for each $1\leq k < \Omega$. The arity of the defining relation $R$ is the set of cardinalities $ar(R) =\{|x| : x\in R\}$.
In this connection, recall that a permutation group $G \leq Sym(\Omega)$ is $k$-\emph{homogeneous} if its action on $k$-element subsets of $\Omega$ is transitive. It is called \emph{set-transitive} if it is $k$-{homogeneous} for every $k\leq |\Omega|$. In this paper we will make use of the description of the set-transitive permutation groups provided in \cite{BP}.

Given a permutation group $G \leq Sym(\Omega)$, a subset $x\subseteq \Omega$ is called \emph{regular}, if the (setwise) stabilizer of the set $x$ in $G$ is trivial. In such a case $G$ acts regularly on the orbit $x^G$ in the power set $P(\Omega)$.





The notion of regular set is closely connected with that of distinguishing number. Recall, that the \emph{distinguishing number}
of a permutation group $G \leq Sym(\Omega)$, denoted $D(G)$, is the least positive integer $k$ such that there exists a partition of $\Omega$ into $k$ parts with the property that only the identity of $G$ stabilizes each part. In particular, $D(G)=2$ if and only if $G$ has a regular set (\cite{DMH}).

For other related papers on regular sets and the distinguishing number see \cite{BC,CNS,glu,SY,SW,tym}.


\section{Transitive imprimitive simple groups}\label{s:ti}

Let $G$ be a permutation group that is simple as an abstract group. As we have mentioned in the introduction, if $G$ is in addition primitive, then by \cite[Theorem~4.2]{SV}, except for alternating groups $A_n$, and two additional groups, $C_5$ and $PSL(2,8)$, each simple primitive group is a relation group (belongs to $BGR(2)$). Thus, in this section, as a first step, we focus on imprimitive transitive groups. We prove the following

\begin{Proposition} \label{p:ti} Each simple transitive imprimitive permutation group $G$ has a regular set, and all subgroups of $G$ belong to $BGR(2)$.
\end{Proposition}

As it is well known, for transitive permutation groups $G$ (considered up to permutation isomorphism), it is enough to study faithful actions of $G$ on its cosets. In this paper by a coset we mean always a \emph{right} coset. By $G/H$ we denote the set of all (right) cosets of a subgroup $H$ in $G$.
The permutation group given by such an action is denoted $(G,G/H)$.
Note that, if $G$ is simple, then each action of $G$ on $G/H$ is faithful and give rise to a transitive permutation group. It is primitive if and only if $H$ is maximal.

Now, given an abstract simple group $G$, let $N<G$. Suppose that $(G,G/N)$ is not primitive. Then there exists a maximal subgroup $H$ of $G$ with $N<H<G$. Thus, the group $(G,G/H)$ is primitive. Elements $\gg \in G$ act faithfully both on $G/N$ and $G/H$, and the first action respects $H$-cosets and determines the second action.
We will say that a subset $x \in P(G/N)$ is compatible with the cosets of $H$, or $H$-compatible, if, for any $\gg\in G$, the condition $x \cap H\gg \neq \varnothing$ implies $x \supseteq H\gg$. For such an $x$, by $\bar{x}$ we denote the corresponding set of cosets of $H$.

We start from a technical lemma, which allows to transfer our results into subgroups of a given group.
It a strengthening of \cite[Lemma~3.1]{SV}.

\begin{Lemma} \label{l:sies}
Let $H\leq S_n$ be a permutation group and $y$ a regular set in $H$. Suppose that $H = \G(R)$ and let $R'$ be the relation obtained from $R$ by deleting the sets of cardinality $|y|$. Let $G = \G(R')$.
If $y^{G} \cap R = \varnothing$, then
every subgroup of $H$ belongs to $BGR(2)$. In particular, this is so, if $|y|\notin ar(R)$.
\end{Lemma}

\begin{proof}
Let $K\leq H$, and let $Q= y^K$. We show that $K= \G(R\cup Q)$. Obviously $K$ preserves $Q$, and as a subgroup of $H$, it preserves $R$. Consequently, $K \subseteq \G(R\cup Q)$.

Conversely, let $\gg\in \G(R\cup Q)$. We show that $\gg \in K$. Let $R''$ be the family of sets in $R$ of cardinality $|y|$. Then, $R\cup Q = R' \cup R'' \cup Q$. The cardinalities of sets in $R'$ are different than $|y|$, so $\gg$ preserves $R'$. It follows that $\gg\in G$. Hence, $\gg$ preserves $y^G$, and since $y^G\cup R = \varnothing$, it preserves $Q \subseteq y^G$. It follows that $\gg$ preserves $R''$, and therefore, $\gg$ preserves $R=R'\cup R''$. This implies that $\gg\in H$.
And finally, since $y$ is regular in $H$, and $\gg$ preserves $Q=y^K$, it follows that $\gg\in K$, as required.
\end{proof}

Now we prove three reduction lemmas. We will use the short notation $BGR(2)^{\#}$ for the class of groups $G\in BGR(2)$ that have a regular set, and such that each subgroup of $G$ belongs to $BGR(2).$

\begin{Lemma}\label{l:NHG:reg} Let $G$ be a simple group with subgroups $N$ and $H$ such that $N < H < G$. If $(G,G/H) \in BGR(2)$ and it has a regular set, then $(G,G/N)\in BGR(2)^{\#}$.
\end{Lemma}

\begin{proof}
Let $R'$ be a relation on $G/H$ such that $\G(R')=(G,G/H)$. Since $(G,G/H)$ is transitive, we may assume without loss of generality that $H\gg \in R'$ for any $\gg\in G\}$. Let $R_0$ be the corresponding relation on $G/N$, that is, consisting of all sets $x\subseteq G/N$ such that there exists a set $x'\in R'$ and $x=\{Ng : g\in G \hbox{\rm\ and } Hg\in x'\}$.
We use it to form a relation $R$ on $G/N$ defining $(G,G/N)$. To this end, let $y$ be an $H$-compatible subset of $G/N$ such that $\bar{y}$ is a regular set of $(G,G/H)$ and $H \notin \bar{y}$. (Such a set exists, since by assumption, $(G,G/H)$ has a regular set, and the complement of a regular set is regular). We define
$R_1$ to consist of all sets of the form $(y\cup \{N\})\hh$ for $\hh \in G$.
Now, we define $R=R_0 \cup R_1$. Since each set in $R_1$ has cardinality not divisible by $[H:N]$, it follows that $ar(R_0) \cap ar(R_1) = \varnothing$. This means that a permutation $u$ preserves $R$ if and only if $u$ preserves each of $R_0$ and~$R_1$.

We show first that $\G(R)\supseteq (G,G/N)$.
We need to show that if $x\in R$, then $x\gg\in R$ for any $\gg \in G$. Indeed, if $x \in R_0$, then this is so, because $\gg$ preserves $R'$; if $x= (y\cup \{N\})\hh \in R_1$ for some $\hh\in G$, then
$x\gg = (y\cup \{N\})(\hh\gg)\in R_1$, as well.

We prove the opposite inclusion.
Let $\uu \in \G(R)$.
By definition of $R$, $\uu$ is a permutation of $G/N$ preserving the cosets of $H$, so we may consider its action on the set $G/H$. In this action, $\uu$ preserves $R'$, which means that there exists $\gg \in \G(R')= (G,G/H)$ whose action on $G/H$ is the same as $\uu$. It is enough to show that $\gg$ has the same action on $G/N$ as $\uu$.

Note that $\ww=\uu\gg^{-1}$ acts as the identity on $G/H$ and it preserves the relation $R$ (by the previous inclusion). Let
$N\hh$ be an arbitrary coset of $N$ with $\hh\in G$.
As $(y\cup \{N\})\hh \in R$, it follows that $(y\cup \{N\})\hh\ww \in R$.
Consequently,
$(y\cup \{N\})\hh\ww = (y\cup \{N\})\hh'$ for some $\hh'\in G$.
Thus, $y\hh\ww \cup \{N\}\hh\ww = y\hh'\cup \{N\}\hh'$.
Since $y$ is regular in $(G,G/H)$, and $w$ acts as the identity on $H$-cosets,
$\hh=\hh'$, and consequently,
$N\hh\ww = N\hh$. Thus, $\ww$ fixes $N\hh$ for any $\hh\in G$, which means that it acts as the identity on $G/N$, proving that $G\in BGR(2)$.

From what we have established it follows also that $y\cup\{N\}$ is a regular set in $(G,G/N)$. Now, the complement of $y\cup\{N\}$ is also regular and its cardinality is not in $ar(R)$, unless $[H:N]=2$, and $|y\cup \{N\}| =[G:N]/2$. In the latter case $y\cup\{N,Ng'\}$, where $Ng'\notin y$ and $g'\notin N$, is also regular and satisfies the assumptions of Lemma~\ref{l:sies}. Thus $G\in BGR(2)^{\#}$, as required.
\end{proof}

In the next lemma we use the idea {that appears in} \cite[Lemma~2.7]{DHM}.

\begin{Lemma}\label{l:NHG:n-1}
Let $G$ be a simple group with subgroups $N$ and $H$ such that $N < H < G$. If $[H:N]\geq [G:H]-1$, then $(G,G/N) \in BGR(2)^{\#}$.
\end{Lemma}
\begin{proof}
Denote $[G:H]=n$, and let $s_1,\ldots,s_{n}$ denotes a partition of $G/N$ into sets corresponding to $H$-cosets of $G$. Then, $|s_i| = [H:N] \geq n-1$.

Let $y$ be a subset of $G/N$ such that for all $i,j \leq n$, if $i\neq j$ then the cardinality $|y\cap s_i|\neq |y\cap s_j|$. Such a set $y$ exists, since $|s_i| =[H:N]\geq n-1$ for all $i$. We may assume that $1$ and $[H:N]$ are among the cardinalities $|y\cap s_i|$, which guarantees that $|y|>|s_i|$. To fix the notation, we assume that it is $s_1$ with $|y\cap s_1|=1$.

We define a relation $R$ as one consisting of all sets $s_i$ and all sets $x= y\hh$ for any $\hh\in G$.
Note that, since $|y|\neq |s_i|$, the arity $ar(R) = \{|s_i|,|y|\}$.

The most important property of $y$ is that as $|y\cap s_i|$ gives an injective labelling of $H$-cosets, $y$ may be used to decode the permutations of $H$-cosets.
More precisely, if $u$ and $w$ are arbitrary permutations of $G/N$ that preserve the partition into sets $s_i$ then the condition $yu=yw$ implies that $u$ and $w$ have the same induced action on the sets $s_i$ ($H$-cosets).

Now, similarly as in the previous proof, it is easy to see that $\G(R) \supseteq (G/N)$. For the opposite inclusion, assume that $\uu\in\G(R)$. Then, since $|y|\neq |s_i|$, $yu = y\hh$ for some $\hh\in G$. By the property of $y$ pointed out above, $\uu$ acts on $H$-cosets as $\hh$.
Hence, $\ww=\uu\hh^{-1}$ acts on $G/H$ as the identity and preserves $R$. Again, all we need to show is that $w$ acts as the identity on $G/N$.

Suppose to the contrary that there exists a coset $\beta \in G/N$ such that $\beta\ww \neq \beta$. Let $\alpha$ denotes the unique $N$-coset in $y\cap s_1$, and let $\gg\in G$ be such that $\alpha\gg = \beta$. Then for $\vv=\gg\ww\gg^{-1}$ we have:
(i) $\alpha\vv \neq \alpha$, (ii) $\vv$ acts as the identity on $H$-cosets, and (iii) $\vv$ preserves $R$. By (iii), $y\vv = y\hh'$ for some $\hh'\in G$. By (ii) and the property of $y$, $y\vv=y$. Consequently, $\vv$ fixes $y \cap s_1 = \{\alpha\}$, which is a contradiction with (i). This proves that $(G,G/N)\in BGR(2)$.

From the proof it follows that $y$ is a regular set in $(G,G/N)$. We show that there exists a point $\alpha=Ng$ such that $y'=y\cup\{\alpha\}$ is regular, as well.
Indeed, if it is possible to add a point to $y$ with keeping the property of injective labeling of $H$-cosets, then we are done. Otherwise, if adding a point always causes that two sets $y\cap s_i$ and $y\cap s_j$, for $i\neq j$, become equinumerous in $y'$ (which happens when $[H:N]=[G:H]=1$), then we can choose an arbitrary pair $i,j$ and add a point in such a way that $|y'\cap s_i|=|y'\cap s_j|=1$, and otherwise the cardinalities $|y'\cap s_i|$ are pairwise different. Such a set $y'$ is regular in $(G,G/N)$, by the same argument as that applied to $y$, unless the transposition $(i,j)\in (G,G/N)$. Since the pair $i,j$ can be chosen arbitrarily, $(G,G/N)$ has a regular set of cardinality $|y|+1 \notin R$ (unless $(G,G/N)=S_n$, which is not the case). Applying Lemma~\ref{l:sies} completes the proof.
\end{proof}

Combining the approaches of the two above lemmas we may get still something more.

\begin{Lemma}\label{l:NHG:d}
Let $G$ be a simple group with subgroups $N$ and $H$ such that $N < H < G$. If $(G,G/H)\in BGR(2)$ and $[H:N]\geq D(G,G/H)-1$, then $(G,G/N) \in BGR(2)$ and has a regular set. If in addition, $D(G,G/H)$ does not divide $[G:H]$ or $[H:N]> D(G,G/H)-1$, then $(G,G/N)\in BGR(2)^{\#}$.
\end{Lemma}
\begin{proof}
We use the ideas of the previous proof.
Denote $d=D(G,G/H)$, and let $s_1,\ldots,s_{n}$ denotes a partition of $G/N$ into sets corresponding to $H$-cosets of $G$. Then, $|s_i| = [H:N] \geq d-1$.
In addition, we consider the partition $t_1,\ldots, t_{d}$ of $s_1,\ldots, s_{n}$ into $d$ sets corresponding to the distinguishing partition of $(G,G/H)$.

Let $c_1, c_2,\ldots,c_d$ be a sequence of pairwise different numbers with $c_1=1, c_d=[H:N] $, and $0\leq c_i \leq [H:N]$, otherwise. The condition $[H:N]\geq D(G,G/H)-1$ guarantees that such a sequence exists. Let $y$ be a subset of $G/N$ such that for each $s_i\in t_j$, $|y\cap s_i|=c_i$. Then
$|y|>|s_i|$.

As in the proof of Lemma~\ref{l:NHG:reg}, let $R_0$ be a relation on $G/N$ corresponding to a relation $R'$ on $G/H$ such that $(G,G/H)= \G(R')$. Again, we assume that $H\gg\in R'$ for all $\gg\in G$.

We define $R=R_0\cup R_1$, where $R_1$ consists of all sets $x= y\hh$ for any $\hh\in G$. We have $ar(R) = ar(R_0) \cup ar(R_1)$.
A problem, comparing with the previous proof, is that this union may not be disjoint. However, all we need is the fact that each permutation $v$ preserving $R$ preserves $R_1$.
Indeed, since all $s_i\in R$ and $|y|>|s_i|$, $v$ preserves $H$-cosets, and since $y\hh$ are the only sets in $R$ that are not the unions of $H$-cosets, $v$ preserves $R_1$.

Further, since $G$ acts faithfully on $G/H$, and $d=D(G,G/H)$ is the distinguishing number for $(G,G/H)$, we have the following property of $y$ (which is enough for our purposes in this case):

$(*)$ if $h\in G$ and $h\neq 1$, then $|y\cap s_i| \neq |yh\cap s_i|$ for some $i\leq n$.

Again, it is easy to see that $\G(R) \supseteq (G/N)$. For the opposite inclusion, assume that $\uu\in\G(R)$. Then, since $\uu$ preserves $R_1$, it acts on $H$-cosets as some $\hh\in G$.
Hence, $\ww=\uu\hh^{-1}$ acts on $G/H$ as the identity and preserves $R$. All we need to show is that $w$ acts as the identity on $G/N$. This part is the same as in the previous proof.

Again, from the proof it follows that $y$ is a regular set in $(G,G/N)$. We wish to modify it so it is still regular, but and satisfies the assumptions of Lemma~\ref{l:sies}. If $[H:N]> D(G,G/H)-1$, then we have room to change one of the numbers $c_i$, $i\neq 1$, so that the resulting set $y'$ has the properties of $y$, but $|y'|\neq |y|$. If $D(G,G/H)$ does not divide $[G:H]$, then the partition $t_1,\ldots,t_d$ of $s_1,\ldots,s_n$ has blocks of different cardinalities, and then it is enough to change the order of $c_1,\ldots,c_d$ to get regular sets satisfying the assumptions
of Lemma~\ref{l:sies},   proving the claim.
\end{proof}

\emph{Remark}. Let us observe that the three lemmas above hold for arbitrary groups $G > H > N$, and the only requirement is that both actions on $G/H$ and $G/N$ are faithful, and $G$ is different from $S_n$. 

We will combine the above lemmas with what is known on primitive groups. In \cite{ser}, Seress lists all primitive groups that have no regular set (cf. \cite[Theorem~2.2]{SV}). Using this list one can distinguish all simple groups in primitive action that have no regular set. In the list below the first entry denotes the degree of the action, and the second---the group itself. If the group is abstractly isomorphic to one of $A_n$ or another group in the list with a different name, then this is indicated. There are no other isomorphisms between the groups in the list except those indicated. (This may be inferred from the classification of the finite simple groups, using, e.g., \cite[Appendix~A]{DM}).
\bigskip

${\mathcal L} = \{(6, L_2(5)\!\cong\! A_5)$,
$(7, L_3(2)\!\cong\! L_2(7))$,
$(8, L_2(7)\!\cong\! L_3(2))$,
$(9, L_2(8))$,

\hspace*{8mm}$(10, L_2(9)\!\cong\! A_6)$,
$(11, L_2(11))$,
$(11, M_{11})$,
$(12, M_{11})$,
$(12, M_{12})$,

\hspace*{8mm}$(13, L_3(3))$,
$(15, L_4(2) \!\cong\! A_8)$,
$(22, M_{22})$,
$(23, M_{23})$,
$(24, M_{24})\}$.
\bigskip

Combining this with \cite[Theorem~4.2]{SV} 
we obtain the following:

\renewcommand{\labelenumi}{(\roman{enumi})}
\begin{Lemma}\label{l:bez_reg}
Let $G$ be a simple primitive permutation group of degree $n\geq 2$ different from $A_n$. Then the following hold:
\begin{enumerate}
\item $G$ has a regular set if and only if $G\notin {\mathcal L}$.
\item $G\in BGR(2)$ if and only if $G\neq C_5, PSL(2,8)$
\end{enumerate}
\end{Lemma}

Now we prove the main result of this section. We use the notation $G^+$ for the permutation group obtained from $G$ by adding an extra fix point. \bigskip

\noindent\emph{Proof of Proposition}~\ref{p:ti}. Let $G$ be a simple permutation group that is transitive but not primitive. Then it is permutation isomorphic to $(G,G/N)$, where $N$ is proper, not maximal subgroup of $G$. Hence there exist a maximal subgroup $H$ of $G$ such that $N < H < G$. Since $H$ is maximal, $(G,G/H)$ is primitive. If $(G,G/H)\in BGR(2)$ and has a regular set then the result holds by Lemma~\ref{l:NHG:reg}. Otherwise, $(G,G/H)$ is either $A_n$ or one of the fifteen groups listed in Lemma~\ref{l:bez_reg}.

First, we consider the case when $(G,G/H)=A_n$. It follows that $H=A_{n-1}^+$, $[G:H]=n$, and for $n> 5$, we have $[H:N']=[A_{n-1}^+:N'] \geq n-1$ for any $N' < A_{n-1}^+$. Then, by Lemma~\ref{l:NHG:n-1}, $(G,G/N)\in BGR(2)$ and has a regular set, as required. If $n=5$, $H=A_4^+$, and $[G:H]=5$. Then, up to conjugation, there are~$4$ proper subgroups $N$ of $H$ (this can be checked using GAP or other systems for computation in permutation groups). For three of these subgroups the index $[H:N] \geq 4$. Hence, we may use Lemma~\ref{l:NHG:n-1} to infer the required claim. The lemma does not apply only in the case when $N$ is of order~$4$. In this case $N=K_4^+$ is the Klein $4$-group (extended by a one fixed point).

We compute this case directly using GAP. Here $(G,G/N)$ is $A_5$ acting on $15$ elements. We can find it in the list of transitive groups as the group $T$ numbered 15T5, generated by permutations
\begin{align*}
g&=(1,9,10,3,14)(2,15,7,12,6)(4,5,11,13,8), \hbox{\rm\ and} \\
h&=(1,4,10)(2,5,8)(3,7,11)(6,9,15)(12,14,13),
\end{align*}
on the set $\Omega =\{1,\ldots,15\}$.
It is imprimitive group with blocks
$$B=\{\{ 1, 6, 8 \}, \{ 2, 4, 9 \}, \{ 3, 7, 11 \}, \{ 5, 10, 15 \}, \{ 12, 13, 14 \}\}.$$
These blocks and the orbits of some sets in $T$ are used to define $T$ by a relation. Let $R_1=\{1,2\}^T$, $R_2= (\Omega\setminus \{1,2,4\})^T$, and let $R=B\cup R_1 \cup R_2$. Then, it is not difficult to check that that $T= G(R)$. It is enough to observe that
$\G(R)$ is contained in $\G(B)$, which is isomorphic to $S_3 \wr S_5$, and to verify computationally that the cardinality of the stabilizer of $R_1 \cup R_2$ in $\G(B)$ is exactly $|A_5|$.
In addition, one can easily check that $T$ has regular sets of sizes from $3$ to $12$. (We have simply asked GAP about the cardinality of the orbits in the action of $G$ on $k$-element subsets, for each $k$. If this cardinality is equal to the order of $G$, it means that each set in the orbit is regular. In other places of this paper we have used this approach, we just give information about cardinalities of regular sets.)

Now, consider the cases when $(G,G/H)$ is one of the fifteen groups listed in Lemma~\ref{l:bez_reg}. First we assume that $(G,G/H) \in BGR(2)$, i.e., is different from $PSL(2,8)$ or $C_5$. We intend to combine Lemma~\ref{l:NHG:d} with
\cite[Theorem~1]{DMH} and Table~2 in Section~3, which shows that $D(G,G/H)\leq 4$ for these groups.
If $D(G,G/H) =2$ or $D(G,G/H) =3$ and $[G:H]$ is not divisible by $3$, then by Lemma~\ref{l:NHG:d}, $(G,G/N)\in BGR(2)^{\#}$, as required.

If $D(G,G/H) =3$ divides $[G:H]$, then \cite[Table~2]{DMH} shows that $(G,G/H)$ is $M_{11}(12)$ acting on 12 elements or $M_{24}$ in the natural action. For $(G,G/H) = M_{11}(12)$, up to conjugation, there is only one subgroup $H$ of $M_{11}$ of index $12$, and all subgroups $N<H$ have index $[H:N]$ at least $11$. By Lemma~\ref{l:NHG:d}, it follows that $(G,G/H)\in BGR(2)^{\#}$. For $(G,G/H) = M_{24}$ there is only one subgroup $H$ of $M_{24}$ of index $24$, and all subgroups $N<H$ have index $[H:N]$ at least $23$. Again, the claim follows by Lemma~\ref{l:NHG:d}.

For $D(G,G/H) =4$ we have only $3$ groups to consider: $PSL(3,2), M_{11}$, and $M_{12}$, and this requires some computation.
\emph{Case 1.} Let $(G,G/H)=PSL(3,2)$. Then, as $H$ we may take any subgroup of index $7$ (since we know that $PSL(3,2)$ has only one primitive action on $7$ elements). Now, for any subgroup $N$ of $H$ of index $[H:N]>2$, by Lemma~\ref{l:NHG:d}, $(G,G/N)\in BGR(2)^{\#}$, as required. Up to conjugation, there is only one subgroup $N$ of $H$ of index $2$. Then, $(G,G/N)$ is a transitive permutation group of degree $14$ abstractly isomorphic to $PSL(3,2)$. This is $T=14{\rm T}10$ imprimitive group generated by permutations
\begin{align*}
g &=(1,5,9,13,3,7,11)(2,6,10,14,4,8,12), \hbox{\rm\ and} \\
h &= (1,10,6,14,11,9,12)(2,5,8,3,13,7,4),
\end{align*}
with blocks
$$B = \{\{1,8\},\{2,9\},\{3,10\},\{4,11\},\{5,12\},\{6,13\},\{7,14\}\}. $$
Now, one can check that for $R=B\cup \{1,2,3,4\}^T$, $T=\G(R)$.
Also, it has regular sets of all sizes from $3$ to $11$, which means that $T\in BGR(2)^{\#}$, as well.

\emph{Case 2.} Let $(G,G/H)=M_{11}$. Then $H$ is the only subgroup of index $11$. By Lemma~\ref{l:NHG:d}, for any subgroup $N$ of $H$ of index $[H:N]>3$, $(G,G/N)\in BGR(2)^{\#}$. Up to conjugation, there is only one subgroup $N$ of $H$ of index less than $4$. It the subgroup $N$ of $H$ of index $2$ isomorphic to $A_6$. Then, $(G,G/N)$ is the transitive permutation group $T=22{\rm T}22$, of degree $22$.

We describe how to obtain a relation defining $T$. First, we take the two-element blocks of imprimitivity, and consider the action of $M_{11}$ on these blocks. Now, $M_{11}=\G(R)$ for some relation $R$ (pointed out in \cite{SV}). We take the induced relation $R'$ on the two-element blocks of $T$ (obtained by replacing each point in $R$ by two points of the corresponding block). Then, all the sets in $R'$ are of an even cardinality. It remains to check, that there exists a $3$-element set whose $T$-orbit together with $R'$ yields a relation generating $T$ (since almost all $3$-element sets have this property, we omit the details). In addition one easily checks that $T$ has regular sets of all sizes from $6$ to $18$, so $5\notin ar(R)$, proves the required claim.

\emph{Case 3.} Let $(G,G/H)=M_{12}$. Since there is only one action of $M_{12}$ on $12$ elements, one may take as $H$ any of the two subgroups of index $11$. As before, by Lemma~\ref{l:NHG:d}, for any subgroup $N$ of $H$ of index $[H:N]>4$, $(G,G/N)\in BGR(2)^{\#}$. Since, there are no subgroup $N$ of $H$ of index less than $11$, the case is completed.

To complete the proof, it remains to consider two groups not in $BGR(2)$.
For $C_5$ we have nothing to prove, since it has no nontrivial subgroups.
For $(G,G/H)=PSL(2,8)$, we apply Lemma~\ref{l:NHG:n-1}, as follows.

First, $H$ is the only subgroup of index $9$. By Lemma~\ref{l:NHG:n-1}, for any subgroup $N$ of $H$ of index $[H:N]\geq 8$, $(G,G/N)\in BGR(2)^{\#}$. Up to conjugation, there is only one subgroup $N$ of $H$ of index less than $8$. This is a group isomorphic abstractly to $C_2^3$. The group $T=(G,G/N)$ is of degree $63$, which makes it hard for a direct computation.

We describe how one can modify the proof of Lemma~\ref{l:NHG:d} to handle this special case. The only assumption that fails to hold is $(G,G/H)\in BGR(2)$. So, rather than the group $(G,G/H)$ itself we consider a construction using two-point stabilizer of $(G,G/H)$, which is in this case $G_{\alpha,\beta} = C_7^{++}$ (the cyclic group on 7 elements with two extra fix points).

So, let $s_1,\ldots,s_{n}$, as before, be a partition of $G/N$ into the sets corresponding to $H$-cosets of $G$, and let $s_1$ and $s_2$ corresponds to the $H$-cosets $\alpha$ and $\beta$. Let
$R'$ be a relation on $G/H \setminus (\alpha\cup \beta)$ such that $C_7=\G(R')$ and corresponding to the stabilizer $G_{\alpha,\beta}$ (see e.g., \cite[p.~385]{kis1} for such a relation). We construct the corresponding relation on $G/N$: for each $x\in R'$ by $x'$ we denote the subset of $G/N$ consisting of all cosets $N\gg$ such that $H\gg\in R'$. In addition, we choose three points: $a\in s_1$ and $b,c\in s_2$, and we put $x'' = x'\cup\{a,b,c\}$. By $R''$ we denote the family of all sets $(x'')\gg$ for all $x\in R'$ and $\gg\in G$. We put $R_0=R'' \cup \{s_1,\ldots,s_{n}\}$. The rest of the proof is the same as that of Lemma~\ref{l:NHG:reg}. Checking that this construction works is left to the reader.
\qed

\section{Subdirect sum of permutation groups}\label{s:ss}
In the next section we consider intransitive simple groups. Here, the basic tool is a result on the subdirect sum of permutation groups \cite{kis1}.
In order to make this paper self-contained we recall here the notion of the \emph{subdirect sum} and the basic result on it.

Given two permutation groups $G \leq Sym(\Omega)$ and $H \leq Sym(\Delta)$, the \emph{direct sum} $G\oplus H$ is the permutation group on the \emph{disjoint}
union $\Omega \cup \Delta$ defined as the set of all permutations $(\gg,\hh)$, $\gg\in G, \hh\in H$ such that
$$
\alpha(\gg,\hh) =
\left\{\begin{array}{ll}
\alpha\gg, & \mbox{if } \alpha\in \Omega\\
\alpha\hh, & \mbox{if } \alpha\in \Delta
\end{array}\right.
$$
Thus, in $G\oplus H$, permutations of $G$ and $H$ act independently in a natural way on the disjoint union of the underlying sets.

We introduce the notion of the \emph{subdirect sum} following \cite{GK1} (and the notion of \emph{intransitive product} in \cite{kis1}).
Let $H_1\lhd\; G_1 \leq S_n$ and $H_2\lhd\; G_2\leq S_m$ be permutation groups such that $H_1$ and $H_2$ are normal subgroups of $G_1$ and $G_2$, respectively. Suppose, in addition, that factor groups $G_1/H_1$ and $G_2/H_2$ are (abstractly) isomorphic and $\phi : G_1/H_1 \to G_2/H_2$ is the isomorphism mapping. Then, by
$$
G = G_1[H_1] \oplus_\phi G_2[H_2]
$$
we denote the subgroup of $G_1 \oplus G_2$ consisting of all permutations $(\gg,\hh)$, $\gg\in G_1, \hh\in G_2$, such that $\phi(H_1\gg) = H_2\hh$.
Each such group will be called the \emph{subdirect sum} of $G_1$ and $G_2$.

If $H_1=G_1$ and $H_2=G_2$, then $G = G_1 \oplus G_2$ is the usual direct sum of $G_1$ and $G_2$.
If $H_1$ and $H_2$ are trivial one-element subgroups, $\phi$ is the isomorphism of $G_1$ onto $G_2$, and the sum is called, in such a case, the \emph{parallel sum} of $G_1$ and $G_2$. Then the elements of $G$ are of the form $(\gg,\phi(\gg))$, $\gg\in G_1$, and both the groups act in a parallel manner on their sets \emph{via} isomorphism $\phi$. In this case we use the notation $G=G_1||_\phi G_2$, where $\phi$ is an (abstract) isomorphism between $G_1$ and $G_2$, or simply $G=G_1|| G_2$ if there is no need to refer to $\phi$.
Note that $G_1$ and $G_2$ need to be abstractly isomorphic, but not necessarily permutation isomorphic, and they may act on sets of different cardinalities.

In the special case when, in addition, $G_1=G_2=G$ and $\phi$ is the identity, we write $G^{(2)}$ for $G||_\phi G$. More generally, for $r\geq 2$, by $G^{(r)}$ we denote the permutation group in which the group $G$ acts in the parallel way (\emph{via} the identity isomorphisms) on $r$ disjoint copies of a set $\Omega$. This group is called
the \emph{parallel multiple} of $G$, and its element are denoted $\gg^{(r)}$ with $\gg\in G$. In particular, we admit $r=1$ and put $G^{(1)}=G$.
For example, the cyclic group generated by the permutation
$\gg = (1,2,3)(4,5,6)(7,8,9)$ is permutation isomorphic to the parallel multiple $C_3^{(3)}$.

The main fact established in \cite{kis1} is that every intransitive group has the form of a subdirect sum, and its components can be easily described. Let $G$ be an intransitive group acting on a set $\Omega = \Omega_1 \cup \Omega_2$ in such a way that $\Omega_1$ and $\Omega_2$ are disjoint fixed blocks of $G$. Let $G_1$ and $G_2$ be restrictions of $G$ to the sets $\Omega_1$ and $\Omega_2$, respectively (they are called also \emph{constituents}). Let $H_1 \leq G_1$ and $H_2 \leq G_2$ be the subgroups fixing pointwise $\Omega_2$ and $\Omega_1$, respectively. Then we have

\begin{Theorem} \cite[Theorem 4.1]{kis1} \label{th:s}
If $G$ is a permutation group as described above, then
$H_1$ and $H_2$ are normal subgroups of $G_1$ and $G_2$, respectively,
the factor groups $G_1/H_1$ and $G_2/H_2$ are abstractly isomorphic, and
$$G = G_1[H_1] \oplus_\phi G_2[H_2],$$
where $\phi$ is an isomorphism of the factor groups.
\end{Theorem}

In this paper, we apply also other results established in \cite{kis1}, and the reader is referred there in case of any doubts. In particular, recall that alternating groups $A_n$, for $n\geq 3$, \emph{are not} the symmetry groups of any $k$-valued Boolean function, i.e., they do not belong to the class $BGR$. Also, the cyclic groups $C_3=A_3, C_4$ and $C_5$ are not $BGR$.

In turn, $S_n\in BGR(2)$ for any $n>1$. Also, for any $k>1$, if groups $G,H\in BGR(k)$, then $G\oplus H \in BGR(k)$, and for any $r\geq 2$, if $H\in BGR(k)$, then $H^{(r)} \in BGR(k)$ \cite[Theorems~3.1 and 4.3]{kis1}. These results will be treated in our proofs as known, without further comments.

There is some subtlety regarding the parallel powers we have to be aware. A well-known fact is that the automorphism group $Aut(G)$ of a group $G$ may have outer automorphisms that are not given by the conjugation action of an element of $G$. In the case of permutation groups $G\leq Sym(\Omega)$ some outer automorphism may still be given by the conjugation action of an element in $Sym(\Omega)$. This corresponds generally to permuting elements of $\Omega$, and such automorphisms are called \emph{permutation automorphisms}. They form a subgroup of $Aut(G)$, which we denote by $PAut(G)$. Often $PAut(G)=Aut(G)$, but some permutation groups have also other automorphisms, which we will call \emph{nonpermutation} automorphisms.

Now, if $G = H\oplus_\psi H$, for some permutation group $H$, where $\psi\in PAut(G)$, then $G$ is permutation isomorphic to $H^{(2)}$ (i.e., $G=H^{(2)}$, according to our convention). If $\psi$ is a nonpermutation automorphism, then $G \neq H^{(2)}$. (More precisely, we should speak here about isomorphisms induced by automorphisms and make distinction between base sets of components, but we assume that this is contained in the notion of the \emph{disjoint union}, and we will make it explicit only when the need arises). An example is the alternating group $A_6$ that has a nonpermutation automorphism $\psi$ (cf. \cite{CL}). Then, $A_6^{(2)}$ and $A_6 ||_\psi A_6$ are not permutation isomorphic.

\section{Intransitive simple groups}\label{s:is}
Without loss of generality we may assume that permutation groups $G$ we consider have no fixed points.
Indeed, if $G$ has fixed points, then $G= G'\oplus I_m$, where $I_m$ is the trivial group on a $m$-element set and $G'$ has no fixed points. All subgroups $H\leq G$ are of the form $H'\oplus I_m$, where $H'\leq G'$. Hence, it is easy to see that $G$ is simple if and only if $G'$ is simple. Moreover, by the results in \cite{kis1} we know that $I_m \in BGR(2)$ for all $m\geq 1$, and $G\in BGR(k)$ whenever $G'\in BGR(k)$ then.
On the other hand, if $G'\notin BGR$, then $G\notin BGR$. So if we prove the dichotomy that either $G'\in BGR(2)$ or $G'\notin BGR$, the results extend immediately on $G$ with fixed points.
Thus, in the rest of the paper, it is assumed tacitly that permutation groups in question have no fixed points.

Our first result describes the structure of simple permutation groups.

\begin{Proposition}\label{p:simple}
Let $G = H[H']\oplus_\phi K[K']$ be an intransitive permutation group which has no nontrivial normal subgroups. Then the groups $H$ and $K$ are simple and isomorphic as abstract groups, and the groups $H'$ and $K'$ are trivial. In particular, $G$ is the parallel sum $G = H||_\phi K$ for some isomorphism $\phi$.
\end{Proposition}
\begin{proof}
First note that the group $H'\oplus I_m$ is a normal subgroup of $G$ (since $H'\lhd H$). As, by our general assumption, $K$ is nontrivial, we infer that $H'$ is trivial. Similarly, we observe that $K'$ is trivial. It follows that $\phi$ is an isomorphism between $H$ and $K$, and $G$ is isomorphic to both $H$ and $K$. Therefore, $H$ and $K$ are simple.
\end{proof}

The proposition above means that intransitive simple permutation groups with nontrivial orbits have always the form of parallel sums. Some remarks are needed to make a proper use of this result.

Note that (using the inverse isomorphism $\phi^{-1}$) we see easily that $G_1||_\phi G_2$ and $G_2 ||_{\phi^{-1}} G_1$ are permutation isomorphic, and since we treat permutation isomorphic groups as identical, the operation of the parallel sum may be considered to be commutative. Further, decomposing each summand step by step we can get a decomposition into transitive components. In particular,
in $G=(G_1 ||_\phi G_2)||_\psi G_3$ all the involved groups must be abstractly isomorphic, and $G$ is permutation isomorphic with $G_1 ||_{\phi'} (G_2||_{\psi'} G_3)$, where isomorphisms $\phi'$ and $\psi'$ are suitably determined by $\phi$ and $\psi$. Thus we may also consider this operation to be associative (up to permutation isomorphism).

So, generally, a simple intransitive permutation group is a parallel sum of two or more transitive components that are all abstractly isomorphic, and the action on the union of orbits is given by a system of suitable isomorphisms between components. (In case, when a group has two different actions on the set of the same cardinality, various systems of isomorphism may lead to different permutation groups; so the pointing out only transitive components may not define the parallel sum uniquely. For example, projective symplectic group $PSp(4,3)$ has two nonequivalent actions on the set of cardinality $n=40$.)

The following lemma gives some sufficient conditions for a parallel sum to be the symmetry group of a Boolean function.

\begin{Lemma}\label{l:reg}
Let $G= H ||_\phi K$ be a parallel sum of two groups. If $H\in BGR(2)$ and it has a regular set $y$, then $G\in BGR(2)^{\#}$.
\end{Lemma}

\begin{proof} Assume that $H\leq Sym(\Omega)$ and $K \leq Sym(\Delta)$ with $\Delta= \{\alpha_1,\alpha_2,\ldots, \alpha_m\}$. Let $R_0$ be a subset of $P(\Omega)$ such that $\G(R_0)=H$, and denote $\Delta_i = \{\alpha_1,\alpha_2,\ldots, \alpha_i\}$.
We define two subsets of $P(\Omega\cup \Delta)$.
\begin{align*}
R_1&=\{x\cup \Delta : \; x\subseteq \Omega,\; |x|= |\Omega|-1\},\\
R_2&=\{ (y\cup \Delta_i)\gg : \; 0<i<m, \; \gg\in G \}.
\end{align*}

Now, we put $R=R_0\cup R_1\cup R_2$. We show that $G=\G(R)$. First, it is easy to see that each relation $R_0,R_1,R_2$ is preserved by permutations in $G$, and therefore $G \subseteq \G(R)$. We prove the opposite inclusion.

Note that sets in $R_0$ are contained in $\Omega$, while each set in $R_1$ or $R_2$ has a nonempty intersection with $\Delta$. Further, each set in $R_1$ contains $\Delta$, which is not the case for any set in $R_2$. It follows that the relations $R_0,R_1,R_2$ are mutually disjoint. Moreover, since the sets in $R_1$ have all cardinality $|\Omega|+|\Delta|-1$, larger than the sets in $R_0$ and $R_2$, it follows that any permutation $\gg$ preserving $R$, preserves $R_1$ itself. Since the elements of $R_1$ are the complements of one-element subsets of $\Omega$, this means that each $\gg\in \G(R)$ preserves the partition into $\Omega$ and $\Delta$.
In particular, each such permutation may be presented as $\gg=\hh\kk$ with $\hh$ and $\kk$ acting on the sets $\Omega$ and $\Delta$, respectively.

If so, $\gg$ preserves also $R_0$ and $R_1$, individually.
In particular, the action of $\G(R)$ on $\Omega$ is contained in the action of $H$, i.e., $\hh\in H$. This proves that $\G(R) \subseteq H \oplus K'$, for some $K'\leq Sym(\Delta)$.
In order to use Theorem~\ref{th:s}, we need to show that
the pointwise stabilizers of $\Omega$ and $\Delta$ in $\G(R)$ are both trivial.

For the first, suppose that $\gg=\hh\kk\in \G(R)$ fixes each point of $\Omega$. This means that $h$ is the identity, $\hh=1$. Thus, for each $i<m$, we have that $(y\cup \Delta_i)\gg = y \cup \Delta_i\kk$. This must belong to $R_2$ (since $R_2$ is preserved by $\gg$), which means that $y \cup \Delta_i\kk = (y\cup \Delta_i)\gg_i$, for some $\gg_i\in G$. The only $\gg_i\in G$ fixing $y$ is $\gg=1$. Therefore, $\Delta_i\kk = \Delta_i$ for all $i<m$, which means that $\kk=1$, as required.

Suppose, in turn, that $\gg=\hh\kk\in \G(R)$ fixes each point $\kappa_i\in \Delta$. Then, for each $i<m$, $(y\cup \Delta_i)\gg = y\hh \cup \Delta_i$, and since $R_2$ is preserved by $\gg$, $y\hh \cup \Delta_i=(y\cup \Delta_i)\gg_i$ for some $\gg_i\in G$. Now,
$(y\cup \Delta_i)\gg_i = y\hh_i \cup \Delta_i\kk_i$ for some $\hh_i\in H$ and $\kk_i \in K$. Thus, for each $i<m$,
$$
y\hh \cup \Delta_i =
y\hh_i \cup \Delta_i\kk_i
$$
Since $\gg_i \in G=H||_\phi K$, $\kk_i$ is uniquely determined by $\hh_i$, and
since $y$ is regular in $H$, $\gg_i=\hh$, for each $i<m$, and $\kk_i = \kk$ does not depend on $i$. Therefore, for each $i<m$, $\Delta_i = \Delta_i\kk$ for some $\kk\in K$. In view of the definition of $\Delta_i$, this implies that $\kk=\kk_i=1$. Consequently (since $\gg_i \in H||_\phi K$), $\hh_i=\hh=1$, as required. By Theorem~\ref{th:s}, $\G(R) = H||_\psi K'$ for some $K'\leq Sym(\Delta)$ and some isomorphism $\psi$ between $H$ and $K'$.

As we have proved that $G \subseteq \G(R)$, we have $H||_\phi K \subseteq H||_\psi K'$, which in view of the finiteness of the groups involved (and the fact that $\psi$ and $\phi$ are bijections) implies $K'=K$, proving that $G\in BGR(2)$.

It remains to prove that $G$ has a suitable regular set, in order to apply Lemma~\ref{l:sies}. First, observe that, since $y$ is regular in $H$, the complement $\Omega\setminus y$ is also regular in $H$. It follows that we may assume that $|y|\neq |\Omega|-1$, taking the complement of $y$ in $\Omega$, if necessary. In turn, since $y$ is regular in $H$, the set $y\cup \Delta$ is regular in $G=H||_\phi K$.
Now, the relation $R'$ obtained from $R$ by deleting the sets of cardinality $|y|+|\Delta|$, contains $R_1$, and therefore $G'=\G(R')$ preserves the partition into $\Omega$ and $\Delta$. It follows that the orbit $(y\cup \Delta)^{G'}$ is disjoint with $R$. Applying Lemma~\ref{l:sies} completes the proof.
\end{proof}

Below we consider the parallel multiple $A_n^{(r)}$ of the alternating group $A_n$. As we will see it belongs to $BGR(2)$ if and only if $r$ is sufficiently large with regard to $n$.

\begin{Lemma}\label{l:An}
For every $n\geq 3$ and $r\geq 1$, the parallel multiple $A_n^{(r)} \in BGR(2)$ if and only if $2^r\geq n$. Moreover, if $2^r\geq n$, then $A_n^{(r)} \in BGR(2)^{\#}$; otherwise, $A_n^{(r)} \notin BGR$.
\end{Lemma}
\begin{proof}
For $r=1$ the claim is that $A_n \notin BGR$, which is the case. So we assume that $r\geq 2$.
First we consider the parallel multiple of $S_n$ and construct a relation $Q$ such that $S_n^{(r)} = \G(Q)$. (By \cite{kis1}, we know that $S_n^{(r)}=BGR(2)$ for any $r\geq 2$, but we need an explicit construction of the set $Q$ to use in the further part of the proof).

The group $S_n^{(r)}$ acts on a set $U$ consisting of $r$ disjoint copies of $\Omega=\{1,2,\ldots,n\}$, and for this proof, we may assume $U = \{(i,j) : 1\leq i \leq r, 1\leq j \leq n \}$, meaning that, for a fixed $i$, the elements $(i,j)$ form the $i$-th copy of $\Omega$.
We define a relation $Q\subseteq P(U)$ such that $S_n^{(r)} = \G(Q)$. We put
$$
Q = \{\{(1,j)\} : 1\leq j\leq n\} \; \cup \; \{\{(i,j),(i+1,j)\} : 1\leq i < r, 1\leq j \leq n\}.
$$

The first set of the union above
contains single elements of the first copy of $\Omega$, while the second set contains suitable unordered pairs.

This guarantees that every permutation in $\G(Q)$ is of the form $\gg^{(r)}$. Indeed, we need to show that for any $i,j$ and any $\hh\in \G(Q)$, the image $(i,j)\hh = (i,j\gg)$ for some permutation $\gg$ of $\Omega$ that does not depend on $i$. Looking for the part containing singletons of $Q$ we infer that $(1,j) = (1,j\gg)$. Next, for the pairs, we have $\{(1,j),(2,j)\}\hh = \{(1,j\gg),(i',j')\}$, and since the only pairs of this form in $R$ are those with $i'=2$ and $j'=j\gg$, we have
$\{(1,j),(2,j)\}\hh = \{(1,j\gg),(2,j\gg)\}$, and consequently, $(2,j)\hh = (2,j\gg)$, as required. The same argument works for further pairs
$\{(i,j),(i+1,j)\}$ with $i=2,\ldots,n-1$, proving the claim.
Thus, we have shown that $\G(\f)\leq S_n^{(r)}$.
On the other hand, obviously, every permutation of the form $\gg^{(r)}$ preserves $Q$, so $\G(Q)= S_n^{(r)}$.

To prove the ``if part'' of the theorem, assume that $2^r \geq n$. We define a regular set $y$ in $S_n^{(r)}$ such that $y\notin ar(Q)$, and apply Lemma~\ref{l:sies}. To this end, let
$j(i)$ denote the $i$-th bit (from the right) in the binary notation of the number $j$. Then we put $(i,j) \in y$ if and only if $j(i)=1$. Thus, the sequences $(1,j),(2,j),\ldots,(r,j)$ are binary representations of different integers provided $2^r \geq n$.

To see that $y$ is regular, note that applying any permutation $\hh\in S_n^{(r)}$ to $y$ corresponds to changing positions of the integers
$(1,j),(2,j),\ldots,(r,j)$, in the parallel way for all $j$. Therefore, all the images $y\hh$ are different.

Since $ar(Q) = \{1,2\}$ we see that $|y| \notin ar(Q)$, provided $|y|>2$. If this is not the case, we replace $y$ by its complement in $U$, which is also a regular set, and satisfies the required condition on the cardinality. Thus, by Lemma~\ref{l:sies}, we infer that every subgroup of $S_n^{(r)}$ belongs to $BGR(2)^{\#}$, as required.

For the ``only if'' part assume that for some $r$ with $2^r < n$ there exists a $k$-valued Boolean function $\f$ such that $\G(\f) \supseteq A_n^{(r)}$. We show that $\G(\f)= S_n^{(r)}$.

Let $y$ be a subset of $U$ with $\f(y)=d$ for some $d<k$ and $\hh$ be an arbitrary permutation in $S_n^{(r)}$. It is enough to show that $\f(y\hh)=d$. If $\hh = \gg^{(r)} \in A_n^{(r)}$, then we are done by the assumption. So assume, in addition, that $\gg \in S_n\setminus A_n$.

As before, consider the sequence $(1,j), (2,j), \dots, (r,j)$ as one determining the binary notation of an $r$-bit number $m=m_j$ according to the condition: $m(i)=1$ if and only if $(i,j) \in y$. Then, since $2^r < n$, there are two positions $s,t < n$ such that $m_s = m_t$, which means that the corresponding sequences for $j=s,t$ are identical. Consequently, there exists a transposition $\uu \in S_n$ such that for $\ww=\uu^{(r)}$ we have $y\ww\gg = y\gg$. Now, $\ww\gg\in A_n^{(r)}$, which means that $\f(y\gg)=d$, as required.

It follows that $A_n^{(r)}\notin BGR(k)$ for any $k\geq 2$, completing the proof.
\end{proof}

In general, if $G$ is a parallel sum of permutation isomorphic components, it needs not to be a parallel multiple. This is so, since permutation groups may have nonpermutation automorphisms. For alternating groups we have an interesting exception.

\begin{Lemma}\label{l:Ann}
Suppose that $G$ is the parallel sum of components permutation isomorphic to a fixed alternating group $A_n$, and $G\neq A_n^{(r)}$ for any $r\geq 1$. Then, $n=6$, $G= A_6^{(r)}||_\psi A_6^{(s)}$ for some $r,s\geq 1$ and $G\in BGR(2)$. Moreover, $G\in BGR(2)^{\#}$, with the exception of $G= A_6||_\psi A_6$ which fails to have a regular set.
\end{Lemma}

\begin{proof}
It is well known \cite{CL} that the only alternating group $A_n$ having any permutation automorphism is $A_6$ and that the index $[Aut(A_6) : PAut(A_6)]=2$, which means that up to permutation automorphisms there is only one nonpermutation automorphism in $A_6$. Hence, $G= A_6^{(r)}$ (which is excluded by our assumption) or $ G=A_6^{(r)}||_\psi A_6^{(s)}$.
We prove that the latter belongs to $BGR(2)$.

First, consider $G= A_6 ||_\psi A_6$.
Using GAP, we find any permutation automorphisms $\psi$ of $A_6$ given by the images of generators of $A_6$:
$$\big( (2,3)(4,5) \big)\psi = (2,5)(3,4), \hbox{\rm\ and \ } \big( (1,2,3,4)(5,6)\big)\psi = (1,2,3,4)(5,6). $$
(Note that this mapping cannot be obtained by permuting points). This may be used to form the group $G=A_6||_\psi A_6$ as one generated by
$$g= (2,3)(4,5)(2',5')(13',14')
\hbox{\rm\ and }h=(1,2,3,4)(5,6)(1',2',3',4')(5',6')$$
on the set $\Omega=\Omega_1\cup\Omega_2$ with $\Omega_1 =\{1,\ldots,6\}$ and $\Omega_2=\{1',\ldots,6'\}$. It is an exercise left to the reader, to check that for $R=\{1,2,3,1'\}^G \cup \{\Omega_1\}$ we have $G=\G(R)$.

Unfortunately, we cannot use Lemma~\ref{l:reg} to complete the proof, because as one can check, $A_6 ||_\psi A_6$ has no regular set. So we still need to consider the group $G' = A_6^{(2)} ||_\psi A_6$. This group, has regular sets of all sizes from $4$ to $20$, and it belongs to $BGR(2)$. The latter can be checked analogously, as the previous case. We take $\Omega' = \Omega_1\cup\Omega_2\cup\Omega_3$, where $\Omega_3=\{1'',\ldots,6''\}$. We form a defining relation for $G'$ using the relation $R$ for $G$. We define
$R'=R \cup \{\Omega\} \cup \{\{1',1''\}, \ldots,\{6',6''\}\}$. It is easy to check that $G'=\G(R')$. Now, using Lemmas~\ref{l:sies} and~\ref{l:reg} completes the proof.
\end{proof}

Now we consider the groups from the list $\mathcal L$ introduced in Lemma~\ref{l:bez_reg}.
We show that although these groups have no regular sets, their sums have. First we prove the following

\begin{Lemma}\label{l:h2}
For each group $H$ in the list $\mathcal L$, and each $r \geq 2$,
$H^{(r)} \in BGR(2)^{\#}$.
\end{Lemma}

\begin{proof}
In view of Lemma~\ref{l:reg}, it is enough to prove the claim for $r=2$.
First assume that $H\neq PSL(2,8)$. Then, by Lemma~\ref{l:bez_reg}, $H\in BGR(2)$. Let $n$ be degree of $H$, $\Omega=\{1,2,\ldots,n\}$, and $\Delta= \{1',2',\ldots,n'\}$. Then, since $H\in BGR(2)$, there is a relation $R_1$ on $\Omega$ with sets of cardinality less than $n$ such that $H=\G(R_1)$.

Let $Q$ be the relation defined in the proof of Lemma~\ref{l:An} for $r=2$ transferred into the set $\Omega\cup \Delta$ by using the natural bijection. The relation $Q$ is to guarantee the parallel action on the sets $\Omega$ and $\Delta$. Yet, the arities of $Q$ and $R_1$ may not be disjoint. Therefore we replace $Q$ by the family $R_2$ of the complements of the sets in $Q$. Then, $ar(R_2) = \{2n-1, 2n-2\}$. Since $2n-2 \geq n$ (as $|\Omega|\geq 2$), $ar(R_1) \cap ar(R_2) = \varnothing $. This guarantees the parallel action on the sets $\Omega$ and $\Delta$, as well.
Hence, putting $R=R_1\cup R_2$, it follows that $\G(R)$ preserves both $R_1$ and $R_2$, which means that $\G(R) = H^{(2)}$, as required.

To apply Lemma~\ref{l:sies}, we need to find a regular set $y$ in $H^{(2)}$ of cardinality $n \leq |y| \leq 2n-3$.
Note that if we find a regular set $z$ of cardinality $3\leq |z| \leq n$, then the set $y = (\Omega\cup \Delta) \setminus z$ is also regular, and has cardinality $2n-3\geq |y| \geq n$, as required. It follows that all we need is to show that in each case there exists a regular set $y$ of cardinality $3 \leq |y| \leq 2n-3$.
This can be done easily with GAP. To provide easy checking for the reader we give some details of computations.

Let us start from the group
$H=L_2(5)$ acting on $n=6$ elements. To present a regular set for $H^{(2)}$ (that the reader can easily check himself or herself) we need to define the underlying set $\Omega$ and generating permutations for $H^{(2)}$. In each case for $\Omega$ we take $\Omega=\{1,2,\dots,n\} \cup\{1',2',\ldots,n'\}$. Then, we find a set of generators for $H$ (one may use, for example, the GAP listing of primitive groups of degree $n$). In this case we get
$\gg = (1,2,5)(3,4,6)$ and $\hh =(3,5)(4,6)$.
We form permutations $\gg^{(2)}$ and $\hh^{(2)}$ obtaining
\begin{align*}
\gg^{(2)} &= (1,2,5)(3,4,6)(1',2',5')(3',4','6), \\
\hh^{(2)} &= (3,5)(4,6)(3',5')(4',6').
\end{align*}
Clearly, the above permutations generate $H^{(2)}$. A required regular set is, for instance, $y=\{1,2,3,8,10,12\}$.

The proof in other cases is the same. The only change is in the choice of permutations $\gg$ and $\hh$ generating $H$, and the regular set $y$. These details are given in Table~\ref{t:1} for each of the cases $H\neq PSL(2,8)$.

\begin{table}[ht!]
\begin{center}
\caption{}
\label{t:1}

\small
\begin{tabular}{ |l|l|l|l| }
\hline
$n$ &$H$ & generators $g$ and $h$ of $H$ & regular set $y$ in $H^{(2)}$ \\ \hline

$7$ & $L_3(2)$ & {$(1,4)(6,7),$} & $\{1,2,3,3',5'\}$ \\
& & {$ (1,3,2)(4,7,5)$} & \\ \hline
$8$ & $L_2(7)$ & {$ (3,7,8)(4,6,5),$} & $\{1,2,2',7'\}$ \\ & & {$ (1,4,2,5,7,8,6)$} & \\ \hline
$10$ & $L_2(9)$ & {$ (1,9,6,3,8)(2,10,7,5,4),$} & $\{1,2,3',5'\}$\\
& & {$ (1,10,6,2,5)(3,4,9,8,7)$} & \\ \hline
$11$ & $L_2(11)$ & {$ (1,5)(2,4)(3,10)(7,11),$} & $\{1,2,3',5',6'\}$\\
& & {$ (3,11,5)(4,7,9)(6,8,10)$} & \\ \hline
$11$ & $M_{11}$ & {$ (1,2,3,4,5,6,7,8,9,10,11)
,$} &
$\{1,2,3,1',5',7'\}$\\
& & {$
(3,7,11,8)(4,10,5,6)
$} & \\ \hline

$12$ & $M_{11}$ & {$ (1,12)(2,10,5,7)(3,8)(4,6,11,9)
,$} &
$\{1,2,3,4,1',4',7'\}$\\
& & {$
(1,3)(2,7)(8,11)(9,10)
$} & \\ \hline
$12$ & $M_{12}$ & {$ (1,4,12,6)(2,7,5,9,8,10,3,11),$} &
$\{1,2,3,4,11,1',2',5',$\\
& & {$
(1,12)(2,6,4,9,7,8,11,3)
$} & $8',9'\}$ \\ 
\hline
$13$ & $L_3(3)$ & {$ (1,10,4)(6,9,7)(8,12,13),$} &
$\{1,2,3,4,1',7',9'\}$\\
& & {$
(1,3,2)(4,9,5)(7,8,12)(10,13,11)
$} & \\ 
\hline

$15$ & $L_4(2)$ & {$ (1,9,5,14,13,2,6)(3,15,4,7,8,12,11),$} &
$\{1,2,3,4,5,6,9,1',$\\
& & {$
(1,3,2)(4,8,12)(5,11,14)(6,9,15)
$} & $7',8',11',15'\}$ \\ 
\hline

$22$ & $M_{22}$ & {$ (1,13,11,17)(2,7)(3,22,12,21)(4,18,16,10)$} & $\{1,2,3,4,7,9,1',3',8'\}$\\ 
& & {$(6,20,19,14)(9,15),$} & \\ 
& & {$ (1,6,12,11,14,5,22)(2,19,16,9,13,21,8)$} & \\
& & {$ (3,17,18,15,7,4,10)$} & \\ \hline

$23$ & $M_{23}$ & {$ (1,2,3,4,5,6,7,8,9,10,11,12,13,14,15,16,$} & $\{1,2,3,4,9,12,1',10',$\\ 
& & {$ 17,18,19,20,21,22,23),$} & $17'\}$\\ 
& & {$ (3,17,10,7,9)(4,13,14,19,5)(8,18,11,12,23)$} & \\
& & {$ (15,20,22,21,16)$} & \\ \hline

$24$ & $M_{24}$ & {$ (1,5)(2,14,7,12)(3,21)(4,17,16,11)$} & $\{1,2,3,5,11,13,19,1'$\\ 
& & {$ (6,20,23,22)(9,10,15,13),$} & $4',14',18',24'\}$\\ 
& & {$ (1,19,15,8,20,23,24,9,14,11,5,10,22,13,2)$} & \\
& & {$ (3,6,4)(7,16,12,17,18)$} & \\ \hline
\end{tabular}
\end{center}
\end{table}

It remains to consider the case when $H=PSL(2,8)$ acting on $n=9$ elements. Since it does not belong to $BGR$, we need to apply another approach. First of all we consider a larger group $S_9^{(2)} \supseteq H^{(2)}$ that belongs to $BGR(2)$. We have $S_9^{(2)}= \G(R_2)$, where $R_2$ is a relation defined in the first part of the proof, for $n=9$ in this case.

As before, using GAP we find generators for $H=PSL(2,8)$. We use
$\gg =(1,5,4,2,8,3,6)$ and
$\hh=(1,8,6,2,7,3,9)$.
As before, we form generators $\gg^{(2)}$ and $\hh^{(2)}$ on the set $\{1,\ldots,9\} \cup \{1',\ldots,9'\}$ for the group $H^{(2)}$.

Now we look for a set $y$ not only having the trivial stabilizer in $G=H^{(2)}$, but also generating $G$ in $S_9^{(2)}$ in the way explained below. We find that the set $y=\{1,2,3,4,2',3',4',5'\}$ is as required. It has the trivial stabilizer in $G$, and denoting $R_1 = y^G$, the orbit of $y$ in $G$, we obtain $G = \G(R)$ for $R=R_1\cup R_2$. We note that $R_1$ and $R_2$ are disjoint (since $ar(R_2)= \{17,16\}$), and the fact that $G = \G(R)$ may be easily checked using GAP.

It remains to observe that the complement of $y$ is also a regular set in $G$, and having $10$ elements, it is disjoint with $R$. This completes the proof.
\end{proof}

We also need to consider subdirect sums of elements from $\mathcal L$ that arise by using a nonpermutation automorphism. 
It is easy to check that the groups in $\mathcal L$ having nonpermutation automorphisms are the following:
$${\mathcal L}^* =\{ L_3(2), L_2(11), L_3(3), L_4(2), M_{12} \}.$$
Moreover, one may also check that the index $[Aut(G):PAut(G)]=2$ in each of these cases. We use this to prove the following.

\begin{Lemma}\label{l:hh}
For a group $H\in \mathcal L$, if $G$ is a parallel sum of $r>1$ copies of $H$, and $G$ is different from $G=H^{(r)}$, then $G\in BGR(2)^{\#}$.
\end{Lemma}
\begin{proof}
Again, in view of Lemma~\ref{l:reg}, it is enough to prove the claim for $r=2$. In view of the remarks above we may restrict to groups $H\in {\mathcal L}^*$ and to one nonpermutation automorphism $\psi$ in each case. We prove the existence
of a relation $R$ such that $G=H||_\psi H=\G(R)$ and it has a regular set $y$ with $|y| \notin ar(R)$. By Lemma~\ref{l:reg}, this implies that $G\in BGR(2)^{\#}$, as required.

As in the proof of Lemma~\ref{l:Ann}, first we find any permutation automorphism of $H$. For $H=L_3(2)$, using GAP, one can find any permutation
automorphism $\psi$ given by the images of generators $$\big( (1,2)(5,7)\big)\psi = (1, 2)(3, 6), \hbox{\rm\ and \ } \big( (2,3,4,7)(5,6)\big)\psi =
(2, 3, 4, 7)(5, 6).$$
Hence, the group $G=H||_\psi H$ on the set $\Omega=\{1,\ldots,7\} \cup \{1',\ldots, 7'\}$ is generated by permutations
$$ g =(1,2)(5,7)(1',2')(3',6') \hbox{\rm\ and }
h= (2,3,4,7)(5,6)(2',3',4',7')(5',6').$$
Now, by Lemma~\ref{l:bez_reg}, there exist a relation $R_0$ on $\{1,\ldots,7\}$ such that $H=\G(R_0)$. Since $H$ is transitive, one may assume that no singleton belongs to $R_0$. Also we may assume that $\{1,\ldots,7\}\notin R_0$. Let $R_0'$ be a copy of $R_0$ on $\{1',\ldots,7'\}$. Let $R=R_0\cup R_0'\cup \{1,\ldots,7\} \cup \{1,1'\}^G$. It is easy to check that $G=\G(R)$ and that $G$ has regular sets of all sizes from $4$ to $10$. Since $ar(R) \subseteq \{2,\ldots,7\}$, the claim follows.

For the next three groups in ${\mathcal L}^*$ the constructions are the same, and the only difference is the nonpermutation automorphism used. Below we give information about regular sets and any permutation automorphism in each of these cases.\smallskip

\noindent a) $L_2(11)$ (regular sets of sizes from $4$ to $18$):
{\small
\begin{align*}
\big((1, 3)(2, 7)(5, 9)(6, 11)\big)\psi &= (1, 4)(2, 3)(5, 10)(9, 11), \\
( (3, 5, 11)(4, 9, 7)(6, 10, 8))\psi &= (3, 5, 11)(4, 9, 7)(6, 10, 8)
\end{align*}}
b) $L_3(3)$ (regular sets of sizes from $6$ to $20$):
{\small
\begin{align*}
\big( (3,5,11)(6,7,9)(8,12,13)\big)\psi &= (3, 8, 7)(5, 12, 9)(6, 11, 13)\\
\big((1,13,7)(2,10,6)(3,5,12)(4,11,9)\big)\psi & = (1, 13, 7)(2, 10, 6)(3, 5, 12)(4, 11, 9)
\end{align*}}
c) $L_4(2)$ (regular sets of sizes from $6$ to $24$.):
\smallskip

{\small
$\big( (1, 9, 5, 14, 13, 2, 6)(3, 15, 4, 7, 8, 12, 11) \big)\psi = (1, 4, 2, 14, 13, 7, 8)(3, 10, 15, 9, 5, 6, 12)$ }
\smallskip

{\small
$\big( (1, 3, 2)(4, 8, 12)(5, 11, 14)(6, 9, 15)(7, 10, 13) \big)\psi = $ }

{\small
\hfill $ = (1, 2, 3)(4, 14, 10)(5, 12, 9)(6, 13, 11)(7, 15, 8)$
}
\smallskip

\noindent d) For $M_{12}$ the construction as above turned out too complex for computations. We have applied another approach, the results of which can be checked easily using GAP.
The group $G = M_{12} ||_\psi M_{12}$ different from $M_{12}^{(2)}$ is generated by $$(1,19,17,23,2,4,8,9)(6,16,21,22)(3,12)(5,24,15,7,18,10,20,13), \hbox{\rm\ and}$$
$$(1,23,8,19,9,6,4,21,16,22,2)(3,14,12,7,11,20,13,5,15,24,10)$$
It is contained in the transitive group $T$ generated by $$(1,16,23,19,9,21,2,4)(5,15,12,10,18,24,14,20)(6,8,22,17)(7,13), \hbox{\rm\ and}$$
$$(1,5)(2,11)(3,22)(4,15)(6,10)(7,23)(8,24)(9,14)(12,21)(13,19)(16,20)(17,18),$$
which in turn, is contained in $M=M_{24}$ generated by
$$(1,5)(2,14,7,12)(3,21)(4,17,16,11)(6,20,23,22)(9,10,15,13)\hbox{\rm\ and}$$
$$(1,19,15,8,20,23,24,9,14,11,5,10,22,13,2)(3,6,4)(7,16,12,17,18).$$
Now, the defining relation for $G$ is $R= \{1,2,3,4,5,6\}^M\cup \{1,2,3\}^T, \{1,2\}^G$. A~regular set in $G$ is $\{1,2,\ldots,10\}.$
\end{proof}

In the next lemma we consider the last case needed to prove our main result.

\begin{Lemma}\label{l:pary}
For each pair of groups $H \cong K$ in the list $\mathcal L$ that are different as permutation groups, the group $H ||_\phi K \in BGR(2)^{\#}$.
\end{Lemma}

\begin{proof}
Using Lemma~\ref{l:bez_reg} we infer that there are exactly five possibilities for $H$ and~$K$:

\renewcommand{\labelenumi}{(\roman{enumi})}
\begin{enumerate}
\item $n=5, 6$ with $A_5 \cong L_2(5)$,
\item {$n=7, 8$ with $L_3(2) \cong L_2(7)$,}
\item $n=6, 10$ with $A_6 \cong L_2(9)$,
\item {$n=11, 12$ with two actions of $M_{11}$,}
\item $n=8, 15$ with $A_8 \cong L_4(2)$,
\end{enumerate}

For each pair we apply the same strategy, similar to that in the proof of Lemma~\ref{l:h2}. The main difference is that in the case of $G=H|| K$ we have no larger group in $BGR(2)$ in the form of parallel sum. Instead we look for the smallest direct sum $H'\oplus K'$ containing $H|| K$ such that we may easily see that $H'\oplus K'$ is the automorphism group of some Boolean function.

Let $H$ and $K$ act on disjoint sets $\Omega_1$ and $\Omega_2$, and let $\Omega=\Omega_1\cup \Omega_2$. We take the least groups $H'$ and $K'$ such that $H\leq H'\leq Sym(\Omega_1)$, $K\leq K'\leq Sym(\Omega_2)$, and $H',K' \in BGR(2)$. Let $R_1$ and $R_2$ be relations on $\Omega_1$ and $\Omega_2$ such that $H'= \G(R_1)$ and $K'= \G(R_2)$. As we have observed in Section~\ref{s:pre}, we may assume that $\Omega_1\in R_1$ or not, and $\Omega_2\in R_2$ or not.
Without loss of generality we assume that $H$ and $K$ are ordered so that $|\Omega_1| \geq |\Omega_2|$, and $\Omega_1\in R_1$, while $\Omega_2\notin R_2$.

Then, $H'\oplus K'$ has a defining relation $R' = R_1\cup R_2$. Indeed, if $\gg\in \G(R')$, then $\gg$ preserves $\Omega_1$, and therefore it preserves $R_2$, as well. Consequently, $\G(R') \subseteq Sym(\Omega_1) \oplus Sym(\Omega_2)$. Since $R_1$ and $R_2$ are defining relations for $H'$ and $K'$, respectively, it follows that $\G(R) =H'\oplus K'$, as required. Moreover, it is clear that for each $m\in ar(R')$, $m\leq |\Omega_1| =\max\{|\Omega_1|,|\Omega_2|\}$.

Next, we look for a set $x\subseteq \Omega$ such that the orbit of $x$ in $G=H|| K$ defines $G$ in $H'\oplus K'$, i.e., such that if we put $R = R'\cup x^G$, then $G=\G(R)$. Having this, the last step is to find a regular set $y$ with $|y|\notin ar(R)$. By the inequality established above, it is enough that $|y|> |\Omega_1|$.

Now, in case (i), we take
$G=L_2(5)|| A_5$, $H'=L_2(5)$, and $K'=S_5$.
As the generators of $L_2(5)\oplus S_5$ we take $\{1,3,4)(2,5,6), (1,2)(3,4), (7,8,9,10,11), (7,8)\}$. As the generators of $G=L_2(5)||A_5$ (formed similarly as the proof of the previous lemma) we take $ (1,3,4)(2,5,6)(8,9,11)$ and
$(1,2)(3,4)(7,8)(9,10)$.

We note that the group $G=L_2(5)|| A_5$ is unique up to permutation isomorphism. This is because $A_5$ has only permutation automorphism (even if $L_2(5)$ has any permutation automorphism). The situation is similar in the remaining cases (one of the groups has only permutation automorphisms), so we make no further mention about it.

We explain the details of finding a set $x$ defining $G$ in $H'\oplus K'$. We make use of the fact that the sets in $R'$ are all contained in one of the orbits defined by $R'$, Therefore any set that has elements in each of the orbits certainly does not belong to $R'$. By trial and error, using GAP, we check the set $x=\{1,3,7,9\}$.
Its cardinality may belong to $ar(R')$, but by the remark above, the orbit $x^G$
is certainly disjoint with $R'$, and therefore to check that $G=\G(R)$, similarly as in the proof of Lemma~\ref{l:sies}, it is enough to check that the stabilizer of the family $x^G$ in $L_2(5)\oplus S_5$ acting on subsets of $\Omega$ is precisely the group $G$. This is easily done using GAP.
Finally, we check that $x$ is actually a regular set in $G$. Therefore, to apply Lemma~\ref{l:reg}, as a
regular set $y$ with $|y|\notin ar(R)$ one may take simply the complement of $x$ in $\Omega$, whose cardinality is $|y|=7 > 6=|\Omega_1|$, as required.

In case (ii), we take
$G=L_2(7)||L_3(2)$, $H'=L_2(7)$, and $K'=L_3(2)$. Generally, in other cases we take $H'=H$ and $K'=K$ except when $H=A_n$, in which case we take $H'=S_n$. This guarantees that in each case $H',K'\in BGR(2)$.

As the generators of $L_2(7)\oplus L_3(2)$ we take the set $$\{(1,6,5)(2,3,7), (9,11,10)(12,15,13),
(1,4)(2,7)(3,5)(6,8), (9,12)(14,15) \},$$ and the set of two generators for $L_2(7)|| L_3(2)$ is formed using suitable products:
$$(1,6,5)(2,3,7)(9,11,10)(12,15,13) \hbox{ and }
(1,4)(2,7)(3,5)(6,8)(9,12)(14,15).$$
As the defining set we take $x=\{1,3,5,10,12,14\}$, and as a regular set $y$ the complement of $x$ in $\Omega$. The other details of computations (to check by the reader) are the same as before. In Table~\ref{t:2}, we list the sets of generators and sets $x$ and $y$ for the remaining three cases.

\begin{table}[ht!]
\begin{center}
\caption{}
\label{t:2}

\small
\begin{tabular}{ |l|l|l| }
\hline
group & generators of $H|| K$ & sets $x$ and $y$ \\ \hline

$L_2(9)||A_6$ & (1,5,3,9,6)(2,7,8,4,10)(11,12,13,14,15), & $x=$\{1,3,13,15\}, \\

& (1,5,3,9,6)(2,7,8,4,10)(11,13,14,15,16) & $y =\Omega\setminus x$\\ \hline
$M_{11}(12)||M_{11}$ & (1,12)(2,10,5,7)(3,8)(4,6,11,9)
& $x =$ \{1,3,7,9,11,14,18\}, \\

&
(13,21,17,19)(16,20,23,18), & $y =\Omega\setminus x$ \\

& (1,3)(2,7)(8,11)(9,10)(14,16)(15,18) & \\

& (19,22)(21,23)& \\ \hline

$L_4(2)||A_8$ & (1,9,5,14,13,2,6)(3,15,4,7,8,12,11) & $x=$\{1,2,20,21\},\\

& (16,17,18,19,20,21,22), & $y=\Omega\setminus$\{1,4,7,8,18,20,23\} \\
& (1,3,2)(4,8,12)(5,11,14)(6,9,15) & \\
& (7,10,13)(21,22,23) & \\\hline

\end{tabular}
\end{center}
\end{table}

We note that in the last entry the set $x$ is not regular in $L_4(2)||A_8$, and therefore, rather than taking the complement of $x$, we find another regular set whose complement is as required. (The regular set found is too large for standard GAP 4.10.1 to compute the stabilizer of the family $x^G$ in $L_4(2)||S_8$; so we have changed a little our approach in this very case).
\end{proof}

\section{Results}

Now, we are ready to prove our main results. Since a simple permutation group $G$ may have fixed points, we will consider the restriction $G'$ of $G$ to those points that are not fixed by $G$.

\begin{Theorem}\label{th:main}
Let $G$ be a simple permutation group, and $G'$ the
restriction of $G$ to the points that are not fixed.
Then $G\in BGR(2)$ if and only if $G'$ is not one of the following:
\begin{enumerate}
\item $G'=A_n^{(r)}$, where $n\geq 3$ $(n\neq 4)$, $r\geq 1$, and
$r< \log_2(n)$;
\item $G'=C_5$ or $G'=PSL(2,8)$.
\end{enumerate}
In cases {\rm (i)} and {\rm (ii)} above, $G\notin BGR$. \end{Theorem}

\begin{proof}
In the proof, we assume that $G=G'$ has no fixed points. The result extends on groups with fixed points by remarks at beginning of
Section~\ref{s:is}.

First, if $G$ is primitive, then using \cite[Theorem~4.2, Corollary~4.3]{SV} we infer that $G\in BGR(2)$, unless it is $A_n, C_5$ or $L_2(8)$. If $G=A_n$ or $C_5$ then we know it does not belong to $BGR$. If $G=PSL(2,8)$ then it is set-transitive (\cite{BP,SV}), so it is not in $BGR$, either.

If $G$ is transitive imprimitive, then the result is by Proposition~\ref{p:ti}: all groups belong to $BGR(2)^{\#}$.

Thus, we may assume that $G$ is intransitive, and by Proposition~\ref{p:simple} (and the following remarks), $G$ is a parallel sum of at least two transitive components, all abstractly isomorphic to $G$.

If none of the components is in $BGR$, then all they are permutation isomorphic to one of the groups pointed out above. Hence, either $G=H^{(r)}$ or $G=H^{(r)}||_\psi H^{(s)}$, with $\psi$ induced by a nonpermutation automorphism of $H$, where $H = A_n (n\neq 4)$, $C_5$, or $PSL(2,8)$.

In the latter case, since $C_5$ and $PSL(2,8)$ have only permutation automorphisms, $H=A_n$, and the claim follows by Lemma~\ref{l:Ann}. (We note, for future reference, that $G\in BGR(2)^{\#}$, with the exception of $A_6||_\psi A_6$).

In the former case,
if $H=A_n$, then by Lemma~\ref{l:An}, $G\in BGR(2)$ if and only if $2^r \geq n$, and otherwise $G\notin BGR$, at all.
If $H=C_5$, then by \cite{CK} (cf. \cite[Theorem 6.1]{kis1}), $H^{r}\in BGR(2)$ for any $r>1$.
Finally, if $H=PSL(2,8)$, then by Lemma~\ref{l:h2}, $G = H^{(r)} \in BGR(2)^{\#}$ for any $r\geq 2$.

Hence, we may assume now that at least one transitive component $H$ of $G$ belongs to $BGR(2)$.
If $H$ has a regular set $y$, then by Lemma~\ref{l:reg}, $G\in BGR(2)^{\#}$.
If $H$ has no regular set, and $G=H^{(r)}$ is a parallel multiple, then $G\in BGR(2)^{\#}$
by Lemma~\ref{l:h2}. If $G$ is a parallel sum of $r>1$ copies of $H$ different from $G=H^{(r)}$, then the same holds by Lemma~\ref{l:hh}.

It remains to consider the situation when $G$ has two transitive components $H$ and $K$ that \emph{are not} permutation isomorphic, but are abstractly isomorphic, $H\in BGR(2)$, and $H$ has no regular set. By Proposition~\ref{p:ti}, $H$ is primitive and belongs to the list $\mathcal L$.
For the other component $K$ we may assume that either it does not belong to $BGR(2)$
(in which case the only possibility is $K=A_n$ for some $n$), or it belongs to $BGR(2)$ and has no regular set (in which case it appears in the list $\mathcal L$). Looking for the list $\mathcal L$ (preceding Lemma~\ref{l:bez_reg}) we see that there are exactly five possibilities for $H$ and $K$, and combining Lemmas~\ref{l:pary} and~\ref{l:reg} we obtain that, in any case, $G\in BGR(2)^{\#}$.
\end{proof}

In the proof above, there is a lot of information on regular sets. Some extra arguments yield the complete description.

\begin{Theorem}\label{th:reg}
Let $G$ be a simple permutation group, and $G'$ the restriction of $G$ to the points that are not fixed.
Then $G$ has a regular set
if and only if $G$ is not one of the following:
\begin{enumerate}
\item $G'=A_n^{(r)}$, where $n\geq 3$ $(n\neq 4)$, $r\geq 1$, and $r < \log_2(n-1)$;
\item $G'$ is primitive and appears in the list $\mathcal L$ in Lemma~{\ref{l:bez_reg}};
\item $G'=A_6 ||_\psi A_6$ is the exceptional group considered in Lemma~\ref{l:Ann}.
\end{enumerate}
\end{Theorem}

\begin{proof} We assume again that $G=G'$ has no fixed points, since it is easy to see that the result extends immediately. The proof follows exactly the proof of the preceding theorem, with one place requiring more consideration.

If $G$ is primitive that the result is by \cite{ser} (cf. \cite[Theorem~2.2]{SV}).
If $G=H^{(r)}||_\psi H^{(s)}$, with $\psi$ induced by a nonpermutation automorphism of $H$, where $H = A_n$, then the claim follows by Lemma~\ref{l:Ann}.

If $G=H^{(r)}$ and $H=A_n$, then we need to show, in addition, that for $n=2^r+1$, $G$ has a regular set, and if $2^r < n-1$, then $G$ has no regular set. Indeed, $G$ acting on a set $\Omega$ has a regular set if and only if there is a partition of $\Omega$ into two sets such that no nontrivial permutation in $G$ preserves this partition. Since $A_n^{(r)}$ acts on $\Omega = \Omega_1 \cup \ldots \cup \Omega_r$, where each $\Omega_i$ is a copy of $\{1,\ldots,n\}$, and the action of $g^{(r)}$ is the same on all copies as the action of $g\in A_n$ on $\{1,\ldots,n\}$, the union of partitions of all $\Omega_i$ into two sets is equivalent to a partition of $\{1,\ldots,n\}$ into $2^r$ sets (or less if some partitions of $\Omega_i$ coincide). Now, if
$n=2^r+1$, then one can choose partitions of the sets $\Omega_i$ equivalent to a partition of $\{1,\ldots,n\}$ consisting of a two element set and $n-2$ singletons. Obviously the only permutation in $A_n$ preserving this partition is the identity. A regular set $y$ in $A_n^{(r)}$ is the union of blocks, one from each partition of $\Omega_i$. If $2^r < n-1$, then any partition of $\{1,\ldots,n\}$ equivalent to partitions of $\Omega_i$ contains either a block with $3$ elements or two 2-element blocks. Then there exists a nontrivial $g\in A_n$ preserving this partition, and therefore, $A_n^{(r)}$ has no regular set.
\end{proof}

Finally, we have also a general result concerning subgroups of simple permutation groups. Note, that $C_5\notin BGR$, and the same concerns its direct sum with $m$ fixed points $C_5\oplus I_m$. Now, $C_5\oplus I_1$ is a subgroup of $PSL(2,5)$, and more generally $C_5\oplus I_{m+1}$ is a subgroup of $PSL(2,5)\oplus I_m$, $m>0$.
This is exceptional with regard to the following.

\begin{Theorem}\label{th:sub}
If a simple permutation group $G$ is a relation group, then every subgroup $H<G$ is a relation group with the exception of
$H=C_5\oplus I_{1} < PSL(2,5)$ or more generally,
$H=C_5\oplus I_{m+1} < PSL(2,5)\oplus I_{m}$ ($m>0$).
\end{Theorem}
\begin{proof}
Again we assume that $G$ has no fixed points and follow the proof of Theorem~\ref{th:main}. The most involved is now the case of primitive groups. By Lemma~3.1 (ii) \cite{SV}, if $G$ is a subgroup of group $H$ that has a regular set, is not set-transitive, and is maximal in $S_n$ with respect to this property, then every subgroup of $H$ (and thus $G$) is a relation group, as required. By \cite{BP} (cf. \cite[Lemma~2.1]{SV}), the only set-transitive simple primitive group not containing $A_n$ is $PSL(2,8)$, which is not a relation group. So, we need to check those simple groups $G$ that are subgroups of primitive groups that have no regular sets. This requires quite substantial computation, which has been performed in \cite{GKh}. By the theorem proved in \cite{GKh}, the only subgroup of a simple primitive group which is not a relation group is just $H=C_5\oplus I_1 < PSL(2,5)$.

Further, following the proof of Theorem~\ref{th:main}, the only case we need still to consider, when it is not proved that $G\in BGR(2)^{\#}$, are those with $G=(C_5)^{(r)}$ or (in view of Lemma~\ref{l:Ann}) $G=A_6 ||_\psi A_6$ with a nonpermutation automorphism $\psi$. In the former case $G$ has no nontrivial subgroups, so there is nothing to prove. In the latter case, it is enough to observe that $A_6 ||_\psi A_6$ is a subgroup of $M_{12}$, whose all subgroups are relational groups by \cite{GKh}.
\end{proof}

\emph{Remark}. In the first part of the proof, for primitive $G$, we have used ideas of the proof \cite[Theorem~4.1]{SV}. Yet, there is a mistake in the proof of this theorem in \cite{SV}. The argument there works only for $|\Omega|>32$. Nevertheless, the theorem is true as it is shown in \cite{GKh}, where a stronger version of it is proved.



\end{document}